\documentclass[preprint,authoryear]{elsarticle}
\usepackage[T1]{fontenc}
\usepackage[utf8]{inputenc}
\usepackage{tabularx}
\usepackage{supertabular}
\usepackage{type1cm}   
\usepackage{booktabs}
\usepackage[authoryear]{natbib}
\usepackage{lmodern}
\usepackage{setspace}
\usepackage{nicefrac}
\usepackage{amsmath,dsfont}
\usepackage{framed}
\usepackage{adjustbox}
\usepackage{a4wide}
\usepackage{amssymb,amsmath,amsthm}
\usepackage{bbm}
\usepackage{changes}
\usepackage[flushleft]{threeparttable}
\definecolor{dblue}{rgb}{0.21,0.21,0.55}

\usepackage{multirow}
\usepackage{lscape}
\usepackage[english]{babel}
\usepackage{graphicx}
\usepackage{arydshln}
\usepackage{rotating}
\usepackage[flushmargin,hang]{footmisc}
\usepackage{amsmath}
\usepackage{lscape}
\usepackage[pdflinkmargin=5pt,pdfstartview={FitBH -32768},plainpages=false]{hyperref}

\DeclareMathAccent{\verywidehat}{\mathord}{largesymbols}{'144}

\allowdisplaybreaks[3]

\newtheorem{theo}{Theorem}

\newtheorem{lem}{Lemma}

\begin{document}
	\renewcommand*{\thefootnote}{\fnsymbol{footnote}}
	
	\title{Asset-specific limit order microstructure noise: Parameter estimation and empirical evidence}
	\author[1]{Markus Bibinger}
	\author[1]{Adrian Grüber\corref{cor1}}
	\author[2]{Moritz Jirak}
	\cortext[cor1]{Corresponding author: Adrian Grüber, email: {adrian.grueber@uni-wuerzburg.de}}
	\address[1]{Faculty of Mathematics and Computer Science, Institute of Mathematics, University of W\"urzburg} 
	\address[2]{Department of Statistics and Operations Research, University of Vienna}

	%
	%
	\begin{frontmatter}
		%
		%
		\begin{abstract}
			{{\normalsize \noindent 
					The one-sided microstructure noise model for high-frequency quotes from a limit order book is generalized to capture asset-specific noise tail behaviour. Estimation of a noise tail parameter becomes the key step for inference. This is possible based on moments of the high-frequency returns when specifying a parametric noise model. For a Gamma noise distribution, we establish a central limit theorem for 1-dependent observations from the resulting Gamma difference distribution and prove that it is valid also in the convolution model with general semimartingale log-price dynamics. We highlight that the noise tail parameter has important implications for estimating the efficient price and its volatility. An empirical analysis of recent NASDAQ limit order book data demonstrates that asset-specific noise tail parameters are relevant in practice. }}
			
			\begin{keyword}
				Gamma difference distribution \sep High-frequency data \sep One-sided noise\\[.25cm]
			\end{keyword}

		\end{abstract}
		%
	\end{frontmatter}

		\thispagestyle{plain}
	\section{Introduction}
	\label{sec:1}
	In the stochastic boundary model from \cite{BJR}, ask quotes $(Y_i)$ in a limit order book are 
	\begin{align}\label{lomn}Y_i=X_{t_i^n}+\epsilon_i\,,\,i=0,\ldots,n,~~\epsilon_i\stackrel{iid}{\sim}F_{\eta},\epsilon_i\ge 0\,,\end{align}
	where high-frequency observations $(X_{t_i^n})_{0\le i\le n}$ over the time interval $[0,1]$ of a semimartingale $(X_t)$ describe a latent efficient log-price. The standard assumption for the exogenous one-sided \emph{limit order microstructure noise (LOMN)} $(\epsilon_i)$, $\epsilon_i\ge 0$, is that its cumulative distribution function (cdf) $F_{\eta}$ satisfies
	\begin{align}\label{noise_dist}
	F_{\negthinspace\eta}(x) =\eta x\big(1+o(1)\big) ,\;\mbox{as}~x\downarrow 0\,.
	\end{align}
	The main interest is inference on $(X_t)$ and its volatility and jumps. The one-sided noise leads to local minima as key statistics for this. Bid quotes are analogously described by upper-bounded noise and their local maxima are used for inference. The optimal estimation of integrated volatility in \cite{BJR} with rate $n^{-1/3}$ is based on local minima (maxima) of ask (bid) quotes on blocks of lengths $h_n\propto n^{-2/3}$, on that the stochastic order of the movement of the semimartingale boundary $h_n^{1/2}$ balances that of blockwise minima (maxima) of one-sided errors $(nh_n)^{-1}$, both being $n^{-1/3}$. The rate $n^{-1/3}$ is faster than the optimal rate for integrated volatility estimation under classical, centered market microstructure noise (MMN) which is known to be $n^{-1/4}$. The theory for LOMN was advanced by \cite{spot} and \cite{jumps} towards spot volatility estimation and jump tests. It was demonstrated that LOMN enables the identification of smaller jumps compared to MMN with a speed advantage of online jump detection and avoiding pulverization effects which impair the inference under MMN. These improved properties make inference based on LOMN and ask and bid quotes appealing compared to MMN applied to mid-quotes or transaction prices and suggest that LOMN exploits the rich information provided by limit order books more efficiently. This paper contributes an extension of \eqref{noise_dist} that better accounts for asset-specific microstructure features.
	%
	\section{Model with Asset-specific Limit Order Microstructure Noise}
	\label{sec:2}
	It is a known stylized fact of high-frequency financial data that microstructural effects are asset-specific, for instance, more liquid stocks tend to have lower noise and lower noise-to-signal ratios in the MMN framework, see, among others, \cite{MMN} and \cite{MMN2}. As in \cite{MMN2}, we consider a parametric model; however, the one-sided structure of LOMN naturally induces a noise tail parameter as the key quantity for its specification.
	
	The accuracy of statistical methods under LOMN heavily depends on the average number of observations in the vicinity of the semimartingale boundary and hence \eqref{noise_dist}. A natural generalization is a cdf with $F_{\negthinspace \eta,\alpha}(x) =\eta x^{\alpha}\big(1+o(1)\big)$, with \emph{noise tail parameter} $\alpha$. It implies the stochastic order $(nh_n)^{-1/\alpha}$ of blockwise minima (maxima) of one-sided errors, a balanced optimal regime with $h_n\propto n^{-2/(\alpha+2)}$, and results in an optimal rate $n^{-1/(\alpha+2)}$ for integrated volatility estimation. This shows that $\alpha$ becomes the key parameter of the LOMN model, making its estimation the primary and most important step for subsequent inference. In particular, asset-specific values of $\alpha$ allow the model to be tailored to assets with differing microstructure. In case that $\alpha$ is close to 0, rates for inference on $(X_t)$ and its characteristics get even close to the setting without noise, while $\alpha>1$ would result in slower rates than in the standard model of \cite{BJR}.
	
	Noise characteristics under MMN are typically estimated from the returns $(Y_i-Y_{i-1})$, exploiting that $(\epsilon_i-\epsilon_{i-1})$ dominates the high-frequency semimartingale increments, see, e.g., Section\ 2.2 of \cite{TSRV}. However, the tail parameter $\alpha$ of $(\epsilon_i)$ is generally not directly related to that of $(\epsilon_i-\epsilon_{i-1})$ and the one-sided structure is lost when taking differences. While \cite{spot} and \cite{jumps} exploit that the distribution of blockwise minima (maxima) can be well approximated based on the behaviour of $(X_t)$, a nonparametric estimation of $\alpha$ would contrarily require order statistics over sufficiently small blocks such that the noise dominates their distribution. Since this approach is theoretically highly involved and its applicability unclear, we consider a parametric noise model that provides a link between the tail parameter and moments preserving the idea of estimating noise characteristics from $(Y_i-Y_{i-1})$. The most natural candidate is the two-parametric Gamma distribution, $\Gamma(\alpha,\beta)$, with density
	\begin{align}\label{GammaPDF}
	f_{\negthinspace\alpha,\beta}(x) = \frac{\beta^\alpha}{\Gamma(\alpha)} x^{\alpha - 1} e^{-\beta x}\mathbf{1}_{(x>0)}\,,
	\end{align}
	whose cdf admits the following Taylor series at the boundary:
	\begin{align}\label{GammaCDF}F_{\negthinspace\alpha,\beta}(x)=\sum_{k=0}^{\infty}\frac{(-1)^k}{k!(\alpha+k)\Gamma(\alpha)}(x\beta)^{\alpha+k}=\frac{\beta^{\alpha}}{\alpha\Gamma(\alpha)} x^{\alpha}\big(1+o(1)\big),~\mbox{as}~x\downarrow 0\,.\end{align}
	The empirical study demonstrates the adequacy of Gamma-based parametric fits.
	%
	\section{Parameter Estimation}
	\label{sec:3}
	We show in the key Lemma \ref{lemma} that the empirical moments of $(Y_i-Y_{i-1})$ are dominated by noise $(\epsilon_i-\epsilon_{i-1})$, such that $(X_t)$ becomes asymptotically negligible for the estimation of $\alpha$. Hence, consider $Z_1,Z_2 \overset{\text{i.i.d.}}{\sim} \Gamma(\alpha,\beta)$ and $Z=Z_1-Z_2$. The distribution of $Z$ is a special case of a Gamma difference distribution (\cite{klar}), which is embedded in the class of variance Gamma distributions (\cite{vargamma}). The distribution does not belong to an exponential family and no finite-dimensional sufficient statistic exists for large sample sizes. Therefore, parameter estimation is intrinsically challenging. Although the model is differentiable in quadratic mean and hence LAN, maximum likelihood estimation in these models is known to be unstable; see Section\ 3.2 of \cite{vargamma} for details and examples. We construct a \emph{method of moments} for this problem. All moments exist and for all $p\in\mathds{N}$:
	\begin{align}\label{moments}
	\mu_{\negthinspace p}:=\mathds{E}[\lvert Z\lvert^p] &= \frac{\Gamma(p+1)}{\Gamma\big(\tfrac{p}{2}+1\big)}\frac{\Gamma\big(\alpha + \tfrac{p}{2}\big)}{\Gamma(\alpha)}\frac{1}{\beta^p} ;
	\end{align}
	see (1.21) in \cite{forrester}. A simple moment estimator based on even moments $p=2,4$ requires only knowledge of the moments of the Gamma distribution and admits an explicit closed-form solution. However, we find that an estimator based on $p=1$ and $p=2$ achieves a substantially lower variance. The identity
	\begin{align}\mu_1^2=g(\alpha)\cdot\mu_2,~\text{with}~g(\alpha):=\frac{2\big(\Gamma\big(\alpha + \tfrac{1}{2}\big)\big)^2}{\pi\Gamma(\alpha)\Gamma(\alpha+1)}
	\end{align}
	yields, given that the inverse $g^{-1}$ exists, 
	\begin{align}
	\hat\alpha:=g^{-1}(\hat\mu_1^2/\hat\mu_2).\label{mm}
	\end{align}
	The function $g:(0,\infty)\to(0,2/\pi)$ is non-negative, differentiable, and strictly increasing, since
	\begin{align*}
	g^\prime(\alpha)= \left(\log(g(\alpha))\right)^\prime \cdot g(\alpha)= g(\alpha)\left(2\psi(\alpha + 1/2) - \psi(\alpha) - \psi(\alpha + 1)\right)>0,
	\end{align*}
	with the Digamma Function $\psi(\alpha) = \frac{\partial}{\partial \alpha} \log \Gamma(\alpha)$, whose strong concavity yields
	\begin{equation*}
	\psi\left(\frac{\alpha}{2} + \frac{\alpha +1}{2}\right)> \frac{1}{2}\psi(\alpha) + \frac{1}{2}\psi(\alpha +1).
	\end{equation*}
	Since we have with Stirling's formula that
	\begin{align*}
	\lim_{\alpha \downarrow 0 }g(\alpha ) &= \lim_{\alpha\downarrow 0}\frac{2(\Gamma(1/2))^2}{\pi\Gamma(\alpha)\Gamma(1)} =	\lim_{\alpha \downarrow 0 }2\alpha =0,\\
	\lim_{\alpha \to \infty} g(\alpha)  &= \lim_{\alpha \to\infty}\frac{2}{\pi}\frac{\Gamma(\alpha + 1/2)}{\Gamma(\alpha)}\frac{\Gamma(\alpha+1/2)}{\Gamma(\alpha +1)}= \frac{2}{\pi},
	\end{align*}
	\eqref{mm} has a unique solution whenever $(\hat{\mu}_1)^2/\hat{\mu}_2 < 2/\pi$. We show asymptotic normality of the estimator in Theorem \ref{theorem} and establish asymptotic confidence intervals. An estimator of $\beta$ is readily obtained by plugging in $\hat\alpha$ in \eqref{moments}.
	\begin{lem}\label{lemma}
		For $\epsilon_i\overset{\text{i.i.d.}}{\sim}\Gamma(\alpha,\beta)$, and $(X_t)$ a semimartingale of finite quadratic variation, with $\Delta_i^n \epsilon=\epsilon_i-\epsilon_{i-1}$ and $\Delta_i^n X=X_{t_i^n}-X_{t_{i-1}^n}$ when $\sup_i|t_i^n-t_{i-1}^n|=O(n^{-1})$, it holds for $p=1$ and $p=2$ that
		\[
		\frac{1}{n}\sum_{i=1}^n \big|\Delta_i^n \epsilon+\Delta_i^n X\big|^p
		=
		\frac{1}{n}\sum_{i=1}^n |\Delta_i^n \epsilon|^p
		+ o_{\mathbb{P}}(n^{-1/2})\,.
		\]
	\end{lem}
	\begin{proof}
	$p=1$: For $0<c<1/2$, set $A_i=\{|\Delta_i^n \epsilon|>n^{-c}\}$. We decompose
	\[
	R_n := \frac{1}{n}\sum_{i=1}^n \big(|\Delta_i^n \epsilon+\Delta_i^n X|-|\Delta_i^n \epsilon|\big)=\frac{1}{n}\sum_{i=1}^n \big(\ldots\big)(\mathbf{1}_{A_i^c}+\mathbf{1}_{A_i})=R_n^{(1)}+R_n^{(2)}.
	\]
	1.\ Based on the convolution integral, or the expansion of the density in Eq.\ (4) of \cite{vargamma}, we get that
	\[\mathbb{P}\big(|\Delta_i^n \epsilon|\le x\big)=O\big(x^{2\alpha\wedge 1}\big),~\mbox{as}~x\downarrow 0\,.\]
	Using this, the triangle inequality and independence of $(X_t)$ and $(\epsilon_i)$, the bound
	\begin{align*}
		\mathbb{E}\big[|R_n^{(1)}|\big]\hspace*{-.05cm}\le \hspace*{-.05cm}\frac{1}{n}\hspace*{-.05cm}\sum_{i=1}^n\mathbb{E}\big[|\Delta_i^n X|\mathbf{1}_{A_i^c}\big]\hspace*{-.05cm}\le\hspace*{-.05cm} \frac{1}{n}\hspace*{-.05cm}\sum_{i=1}^n \mathbb{E}\big[|\Delta_i^n X|\big] \mathbb{P}\big(|\Delta_i^n \epsilon|\le n^{-c}\big)\hspace*{-.05cm}=\hspace*{-.05cm}O\big(n^{-1/2-c(2\alpha\wedge 1)}\big)
	\end{align*}
	yields with Markov's inequality that $R_n^{(1)}=o_{\mathbb{P}}(n^{-1/2})$.\\
	2.\ For $\delta>0$ with $c+\delta<1/2$, define \(B_i:=\{|\Delta_i^n X|>n^{-c-\delta}\}\).
	On $A_i\cap B_i^c$,
	\(
	|\Delta_i^n X|\le n^{-c-\delta}<n^{-c}<|\Delta_i^n \epsilon|,
	\)
	hence
	\[
	|\Delta_i^n X+\Delta_i^n \epsilon|-|\Delta_i^n \epsilon|
	=
	U_i\,\Delta_i^n X,
	\qquad
	U_i:=\mathrm{sgn}(\Delta_i^n \epsilon),
	\]
	and since $(U_i)$ is centered, $1$-dependent, and independent of $(X_t)$:
	\[
	\frac{1}{n}\sum_{i=1}^n
	\big(\cdots\big)\mathbf{1}_{A_i\cap B_i^c}
	=
	\frac{1}{n}\sum_{i=1}^n
	U_i\,\Delta_i^n X\,\mathbf{1}_{A_i\cap B_i^c}=O_{\mathbb P}(n^{-1})\,.
	\]
	On $B_i$, the triangle inequality and a simple estimate yield
	\[\frac{1}{n}\Big|\sum_{i=1}^n
	\big(|\Delta_i^n X+\Delta_i^n \epsilon|-|\Delta_i^n \epsilon|\big)
	\mathbf{1}_{ B_i}\Big|\hspace*{-.05cm}\le \hspace*{-.05cm}\frac{1}{n}\hspace*{-.05cm}\sum_{i=1}^n|\Delta_i^n X|
	\mathbf{1}_{ B_i}\hspace*{-.05cm}\le n^{c+\delta-1}\hspace*{-.05cm}\sum_{i=1}^n|\Delta_i^n X|^2\hspace*{-.05cm}=\hspace*{-.05cm}O_{\mathbb{P}}(n^{-1+c+\delta}),
	\]
	such that the last two estimates yield $R_n^{(2)}=o_{\mathbb{P}}(n^{-1/2})$.\\ 
	$p=2$: This simpler case readily follows from the decomposition
	\begin{align*}
		\frac{1}{n}\sum_{i=1}^n \big(\Delta_i^n \epsilon+\Delta_i^n X\big)^2&=\frac{1}{n}\sum_{i=1}^n \big(\Delta_i^n \epsilon\big)^2+\frac{1}{n}\sum_{i=1}^n \big(\Delta_i^n X\big)^2+\frac{2}{n}\sum_{i=1}^n \Delta_i^n \epsilon\, \Delta_i^n X\\
		&=\frac{1}{n}\sum_{i=1}^n \big(\Delta_i^n \epsilon\big)^2+O_{\mathbb{P}}(n^{-1}),
	\end{align*}
	where we use that $(X_t)$ has finite quadratic variation and that the cross term has variance
	\begin{align*}
		\text{Var}\left(\frac{2}{n}\sum_{i=1}^n \Delta_i^n \epsilon\, \Delta_i^n X\right)=\frac{4}{n^2}\sum_{\substack{1\le i,j\le n\\ |i-j|\le 1}}^n \mathds{E}\big[\Delta_i^n \epsilon\, \Delta_j^n \epsilon\big] \mathds{E}\big[ \Delta_i^n X\,\Delta_j^n X\big]=O(n^{-2}),
	\end{align*}
	by independence of $(X_t)$ and $(\epsilon_i)$ and the i.i.d.\ property of $(\epsilon_i)$ and is centred, such that Chebyshev's inequality yields the result.
	\end{proof}
	The lemma shows that estimator \eqref{mm} works in our model \eqref{lomn} even with quite general semimartingales $(X_t)$ without further restrictions on their jumps, volatility and drift characteristics. Observation times need not to be equidistant, the condition $\sup_i|t_i^n-t_{i-1}^n|=O(n^{-1})$ is standard when considering inference on $(X_t)$. We use it in 1.\ of the proof for $p=1$, the other steps even work for $\sup_i|t_i^n-t_{i-1}^n|\to 0$.
	\begin{theo}\label{theorem}
		The estimator \eqref{mm} applied to the data \eqref{lomn} satisfies
		\begin{align}\label{clt}
			\sqrt{n}\big(\hat{{\alpha}} - \alpha \big)  \overset{d}{\to}\mathcal{N}\big( 0,\operatorname{AVAR}(\alpha)\big)\,,
		\end{align}
		with $\operatorname{AVAR}(\alpha)$ given in \eqref{avar}.
	\end{theo}
	\begin{proof}
	Since for $Z_1\sim \Gamma(\alpha,\beta)$, it holds that $\beta Z_1\sim \Gamma(\alpha, 1)$, it suffices to consider $Z_i\overset{\text{i.i.d.}}{\sim} \Gamma(\alpha,1)$ in this proof. We establish the bivariate central limit theorem 
	\begin{align}\label{cltb}
		\sqrt{n}\left(\begin{pmatrix}\frac{1}{n}\sum_{i=1}^{n} \lvert Z_i - Z_{i-1}\lvert \\[.1cm] \frac{1}{n}\sum_{i=1}^{n}\lvert Z_i - Z_{i-1}\lvert ^2	\end{pmatrix} - \begin{pmatrix}\mu_1 \\[.1cm] \mu_2	\end{pmatrix}\right) \overset{d}{\to} \mathcal{N}\left(0,\Sigma\right)\,,
	\end{align}
	where $\Sigma=(\Sigma_{i,j})_{1\le i,j\le 2}$ has the generic illustration
	\begin{align*}
		\Sigma_{11} &= \text{Var}\left(\lvert  Z_2- Z_1\lvert \right) + 2\text{Cov}(\lvert Z_2 -Z_1 \lvert, \lvert Z_3 - Z_2\lvert),\\
		\Sigma_{22} &= \text{Var}\left(\lvert  Z_2 - Z_1\lvert^2 \right) + 2\text{Cov}(\lvert  Z_2 - Z_1 \lvert^2, \lvert  Z_3 - Z_2\lvert^2),\\
		\Sigma_{12}=\Sigma_{21} &= \text{Cov}(\lvert Z_2 - Z_1\lvert,\lvert Z_2 - Z_1\lvert ^2)+ 2 \text{Cov}(\lvert Z_3-Z_2 \lvert, \lvert Z_2 - Z_1\lvert ^2).
	\end{align*}
	A first trick is \(Z_i - Z_{i-1}=\tilde Z_i - \tilde Z_{i-1}\), where $\tilde Z_i=Z_i-\mathds{E}[Z_i]$. For $\Sigma_{22}$ this yields with moments of the $\Gamma(\alpha,1)$ distribution that $\Sigma_{22}=4 \mathds{E}[\tilde Z_1^4]=12\alpha(\alpha+2)$. The terms $\text{Var}(\vert  Z_2- Z_1\vert)=\mu_2-\mu_1^2$, and
	\[\text{Cov}(\lvert Z_2 - Z_1\lvert,\lvert Z_2 - Z_1\lvert ^2) =\mu_3-\mu_1\cdot \mu_2 = \frac{4(\alpha+1)\Gamma(\alpha + 1/2)}{\sqrt{\pi} \Gamma(\alpha)}\,\]
	are readily computed inserting the moments \eqref{moments} with $\beta=1$. For the two remaining covariance terms, we use the tricks that:
	\begin{align}\label{trick17}x^p f_{\negthinspace \alpha}(x)&=
		\left(\prod_{k=0}^{p-1} (\alpha + k)\right) f_{\negthinspace \alpha+p}(x),~\mbox{for}~p\in\mathds{N},\\
		\label{trick18}F_{ \alpha+1}(x)&=F_{\alpha}(x)-f_{\alpha+1}(x)\,,
	\end{align}
	writing $f_{\negthinspace \alpha}=f_{\negthinspace \alpha,1}$ and $F_{\negthinspace \alpha}=F_{\negthinspace \alpha,1}$. This is used to simplify integrals as
	\begin{align}
		&\notag \int_{0}^{\infty} \lvert z_2 - z_1 \lvert f_{\negthinspace \alpha}(z_1) \mathrm{d}z_1 = \int_{0}^{z_2}(z_2 - z_1) f_{\negthinspace \alpha}(z_1)\mathrm{d}z_1+ \int_{z_2}^{\infty}(z_1 - z_2) f_{\negthinspace \alpha}(z_1) \mathrm{d}z_1 =\\
		&\notag =z_2 F_{\negthinspace \alpha}(z_2)  - \alpha F_{\negthinspace \alpha+1}(z_2) + \alpha (1-F_{\negthinspace \alpha+1}(z_2)) - z_2 ( 1-F_{\negthinspace \alpha}(z_2)) \\
		&= z_2 (2F_{\negthinspace \alpha} (z_2) -1) - \alpha (2F_{\negthinspace \alpha+1}(z_2) -1)\,.\label{inner}
	\end{align}
	Conditioning on $Z_2$, and using independence of $Z_1$ and $Z_3$, we obtain 
	\begin{align}\label{inserted}
		\mathds{E}[\lvert Z_2 -Z_1 \lvert \lvert Z_3 - Z_2\lvert] &= \int_{0}^{\infty}\hspace*{-0.05cm}\int_{0}^{\infty} \hspace*{-0.05cm}\int_{0}^{\infty}\lvert z_2 -z_1 \lvert\cdot \lvert z_3 - z_2\lvert f_{\negthinspace \alpha}(z_1)f_{\negthinspace \alpha}(z_2) f_{\negthinspace \alpha}(z_3) \mathrm{d}z_1 \mathrm{d}z_2 \mathrm{d}z_3\\
		&=\int_{0}^{\infty} \left(\int_{0}^{\infty} \lvert z_2 - z_1 \lvert f_{\negthinspace \alpha}(z_1)\mathrm{d}z_1\right)^2 f_{\negthinspace \alpha}(z_2) \mathrm{d}z_2\,.
	\end{align}
	Inserting \eqref{inner} for the inner integral, the most complex term is
	\begin{align*}
		\int_0^\infty\hspace*{-.1cm}
		\big(F_{\negthinspace\alpha}(x)\big)^2
		f_{\negthinspace\alpha+2}(x)\,\mathrm{d}x&=
		\left[
		\big(F_{\negthinspace\alpha}(x)\big)^2
		F_{\negthinspace\alpha+2}(x)
		\right]_0^\infty
		-
		2\int_0^\infty\hspace*{-.1cm}
		F_{\negthinspace\alpha}(x)
		F_{\negthinspace\alpha+2}(x)
		f_{\negthinspace\alpha}(x)\,\mathrm{d}x\\
		&=1
		-
		2\int_0^\infty\hspace*{-.1cm}
		F_{\negthinspace\alpha}(x)
		\Bigl(
		F_{\negthinspace\alpha}(x)
		-
		f_{\negthinspace\alpha+1}(x)
		-
		f_{\negthinspace\alpha+2}(x)
		\Bigr)
		f_{\negthinspace\alpha}(x)\,\mathrm{d}x\\
		&=1/3
		+
		2\int_0^\infty\hspace*{-.1cm}
		F_{\negthinspace\alpha}(x)
		\Bigl(
		f_{\negthinspace\alpha+1}(x)
		+
		f_{\negthinspace\alpha+2}(x)
		\Bigr)
		f_{\negthinspace\alpha}(x)\,\mathrm{d}x\,,
	\end{align*}
	using partial integration and that for any cdf $F$ and $V\sim F$, $F(V)$ is uniformly distributed on $(0,1)$, with second moment 1/3.
	For one remaining term, we use the identity
	\[
	f_{\negthinspace\alpha+1}(x)f_{\negthinspace\alpha}(x)
	=
	\frac{\Gamma(2\alpha)}
	{\alpha 2^{2\alpha}\Gamma(\alpha)^2}
	f_{2\alpha,2}(x),
	\]
	such that with \(I_x(a,b)\) the regularized incomplete beta function:
	\[
	\int_0^\infty\hspace*{-.1cm}
	F_{\negthinspace\alpha}(x)
	f_{\negthinspace\alpha+1}(x)
	f_{\negthinspace\alpha}(x)\,\mathrm{d}x
	=
	\frac{\Gamma(2\alpha)}
	{\alpha 2^{2\alpha}\Gamma(\alpha)^2}
	\mathbb{P}\!\left(Z_1\leq \frac{Z_2+Z_3}{2}\right)=\frac{\Gamma(2\alpha)}
	{\alpha 2^{2\alpha}\Gamma(\alpha)^2}I_{1/3}(\alpha,2\alpha).
	\]
	The second remaining term can be reduced similarly using
	\[
	f_{\negthinspace\alpha+2}(x)
	f_{\negthinspace\alpha}(x)
	=
	\frac{\alpha}{\alpha+1}
	\big(f_{\negthinspace\alpha+1}(x)\big)^2 \,.
	\]
	This is, however, not required, since this term is compensated by another one and drops out with parts of the term
	\[
	\int_{0}^{\infty}\hspace*{-.1cm} F_{\alpha}(x)F_{\negthinspace\alpha+1}(x) f_{\alpha+1}(x) \,\mathrm{d}x = \frac{1}{3} + \int_{0}^\infty\hspace*{-.1cm} F_{\alpha}(x) \big(f_{\negthinspace\alpha+1}(x)\big)^2 \,\mathrm{d}x - \int_{0}^\infty\hspace*{-.1cm} \big(f_{\negthinspace\alpha+1}(x)\big)^3 \,\mathrm{d}x.
	\] 
	The last remaining technically more involved term can be simplified to
	\[
	\int_{0}^{\infty}\hspace*{-.1cm}(F_{\negthinspace\alpha+1}(x))^2 f_{\negthinspace\alpha}(x) \,\mathrm{d}x = 1/3 + \int_{0}^{\infty}\hspace*{-.1cm}(f_{\negthinspace\alpha+1}(x))^2 f_{\negthinspace\alpha}(x) \,\mathrm{d}x  - 2 \int_{0}^{\infty}\hspace*{-.1cm}F_{\negthinspace\alpha}(x) f_{\negthinspace\alpha+1}(x) f_{\negthinspace\alpha}(x) \,\mathrm{d}x.
	\]
	
	The remaining terms of the form $\int_{0}^{\infty}\hspace*{-.1cm}F_{\negthinspace\alpha+k}(x) f_{\negthinspace\alpha+l}(x) \,\mathrm{d}x$ for arbitrary $k,l \in\mathbb N$ can be calculated analogously using partial integration, \eqref{trick18} and that $F(V)$ is uniformly distributed on $(0,1)$ for $V\sim F$. Collecting all terms in \eqref{inserted} yields
	\[
	\mathbb{E}\bigl[|Z_2-Z_1||Z_3-Z_2|\bigr]
	=
	\frac{\alpha}{3}
	+
	\frac{4\Gamma(3\alpha)}
	{3^{3\alpha-1}\Gamma(\alpha)^3}
	+
	\frac{\Gamma(2\alpha)}
	{2^{2\alpha-2}\Gamma(\alpha)^2}
	\left(
	2I_{1/3}(\alpha,2\alpha)-1
	\right)\,,
	\]
	what results in
	\begin{align}
		\Sigma_{11}
		=
		\frac{8\alpha}{3}
		-
		\frac{12\Gamma(\alpha+\frac{1}{2})^2}
		{\pi\Gamma(\alpha)^2}
		+
		\frac{8\Gamma(3\alpha)}
		{3^{3\alpha-1}\Gamma(\alpha)^3}
		+
		\frac{\Gamma(2\alpha)}
		{2^{2\alpha-3}\Gamma(\alpha)^2}
		\left(
		2I_{1/3}(\alpha,2\alpha)-1
		\right).
	\end{align}
	For $\Sigma_{12}$, we insert \eqref{inner} and get
	\begin{align*}
		\mathbb E[\lvert Z_2 -Z_1 \lvert \cdot Z_1] &= \alpha(\alpha+1) \int_{0}^\infty(2F_{\alpha}(x) - 1)f_{\alpha+2}(x) \,\mathrm{d}x - \alpha^2\int_{0}^{\infty}(2F_{\alpha+1}(x)-1)f_{\alpha+1}(x)  \,\mathrm{d}x \\
		&= 2(2\alpha+1)\cdot\frac{\Gamma(2\alpha)}{2^{2\alpha}\Gamma(\alpha)^2},
	\end{align*}
	as well as
	\begin{align*}
		&\mathbb E[\lvert Z_2 -Z_1 \lvert \cdot Z_1^2] = \alpha(\alpha+1)(\alpha+2) \int_{0}^\infty(2F_{\alpha}(x) - 1)f_{\alpha+3}(x) \,\mathrm{d}x -\\
		& - \alpha^2(\alpha+1)\int_{0}^{\infty}(2F_{\alpha+1}(x)-1)f_{\alpha+2}(x)  \,\mathrm{d}x=2(\alpha+2)(2\alpha +1) \frac{\Gamma(2\alpha)}{2^{2\alpha}\Gamma(\alpha)^2},
	\end{align*}
	resulting in
	\begin{align}
		\Sigma_{12}=
		4(\alpha+1)
		\frac{\Gamma\left(\alpha+\frac{1}{2}\right)}
		{\sqrt{\pi}\Gamma(\alpha)}
		+
		4(\alpha+2)
		\frac{\Gamma(2\alpha)}
		{2^{2\alpha}\Gamma(\alpha)^2}=
		4(3\alpha+4)
		\frac{\Gamma(2\alpha)}
		{2^{2\alpha}\Gamma(\alpha)^2}.
	\end{align}
	Based on this covariance structure, and since the existence of moments readily implies a Lyapunov condition, a classical central limit
	theorem, e.g., \cite{berk}, applies to linear combinations of $\sum_i|Z_i-Z_{i-1}|$ and $\sum_i|Z_i-Z_{i-1}|^2$ with 1-dependent addends. Hence, by Cram\'er--Wold's device, we conclude \eqref{cltb}. Standard applications of the delta method yield \eqref{clt} with
	\begin{align}\label{avar}\operatorname{AVAR}(\alpha)=
		\left(
		\frac{\pi\Gamma(\alpha)^2}
		{\Gamma(\alpha+\frac{1}{2})^2}\Sigma_{11}
		-
		\frac{\sqrt{\pi}\Gamma(\alpha)}
		{\alpha\Gamma(\alpha+\frac{1}{2})}\Sigma_{12}
		+
		\frac{\Sigma_{22}}{4\alpha^2}
		\right)
		\left(\left(\log(g(\alpha))\right)^{\prime}
		\right)^{-2}\,,
	\end{align}
	where $\left(\log(g(\alpha))\right)^{\prime}=2\psi(\alpha \hspace*{-0.05cm}+\hspace*{-0.05cm} 1/2) \hspace*{-0.05cm}-\hspace*{-0.05cm} \psi(\alpha) \hspace*{-0.05cm}-\hspace*{-0.05cm} \psi(\alpha \hspace*{-0.05cm}+ \hspace*{-0.05cm}1)$ and the first factor becomes
	\begin{align*}
		\frac{8\pi\alpha\Gamma(\alpha)^2}
		{3\Gamma(\alpha+\frac{1}{2})^2}
		-15
		+
		4\sqrt{3}
		\frac{\Gamma(\alpha+\frac{1}{3})
			\Gamma(\alpha+\frac{2}{3})}
		{\Gamma(\alpha+\frac{1}{2})^2}
		+
		\frac{4\sqrt{\pi}\Gamma(\alpha)}
		{\Gamma(\alpha+\frac{1}{2})}
		\left(
		2I_{1/3}(\alpha,2\alpha)-1
		\right)
		-\frac{2}{\alpha}\,.
	\end{align*}
	A general parameter $\beta$ cancels out in \eqref{avar}. The result hence transfers to $(\epsilon_i)$ and by Lemma \ref{lemma} further to the estimator \eqref{mm} when applied to $(Y_i)$ from \eqref{lomn}.
	\end{proof}
	\begin{figure}[b]
		\centering
		\includegraphics[width=7.78cm]{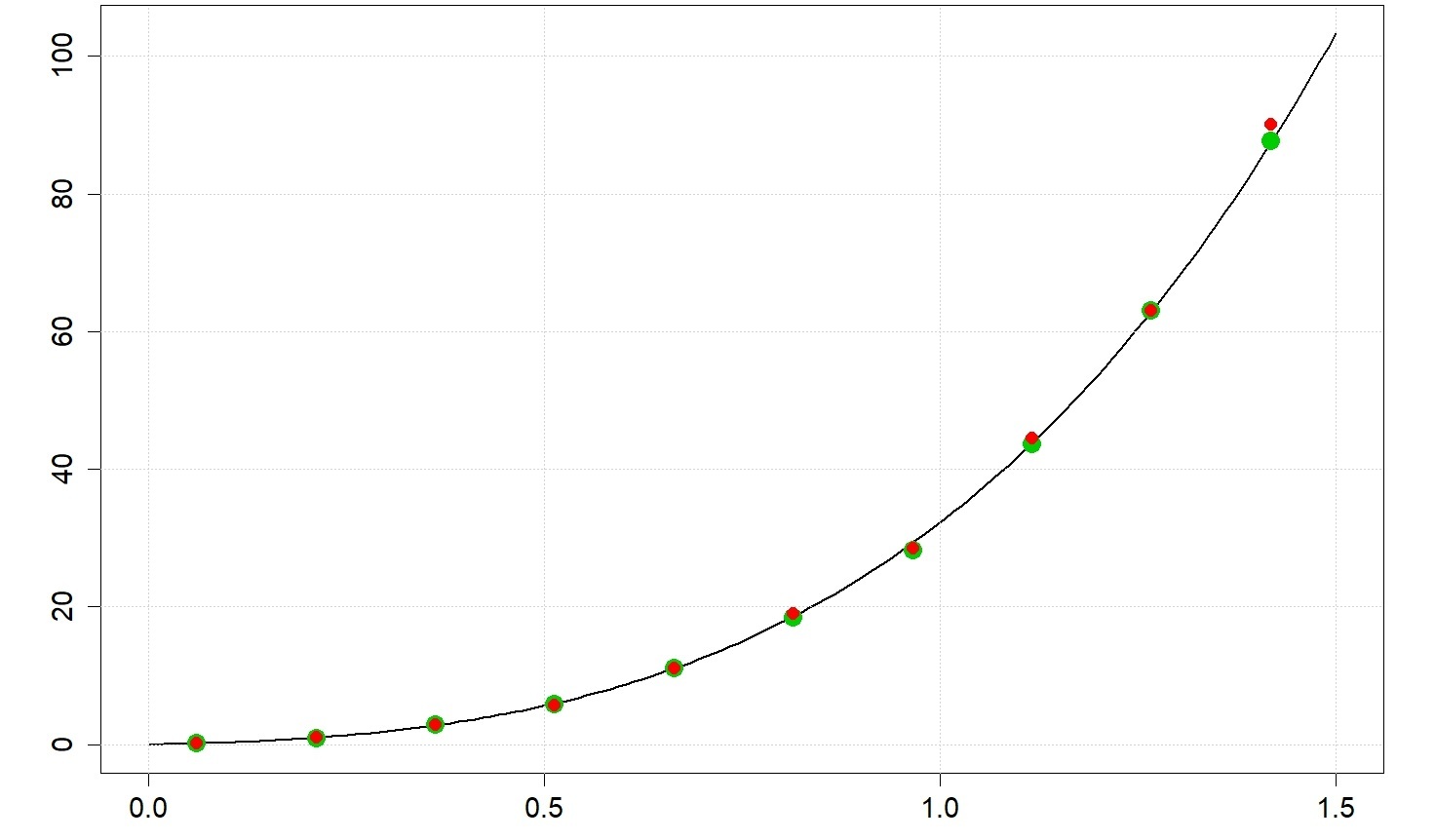}
		
		\caption{The line shows the $\operatorname{AVAR}$ function from \eqref{avar} as a function of $\alpha$. The red and green points compare to Monte Carlo empirical variances of $\hat\alpha$ with sample sizes $n=10,000$ (red) and $n=50,000$ (green).}
		\label{fig:1}
	\end{figure}
	The $\operatorname{AVAR}$ does not depend on $\beta$. It is an increasing, continuous function illustrated in Figure \ref{fig:1}. Hence, plug-in of $\hat\alpha$ facilitates asymptotic confidence intervals by Slutsky's lemma. This allows moreover testing the null hypothesis $\alpha=1$ of standard LOMN. For i.i.d.\ observations from a Gamma difference distribution, the $\operatorname{AVAR}$ would simplify, as the more involved covariance terms vanish, while asymptotic normality remains valid. 
	\section{Simulation Study}
	\label{sec:4}
	\begin{figure}[t]
		\centering
		\includegraphics[width=6cm]{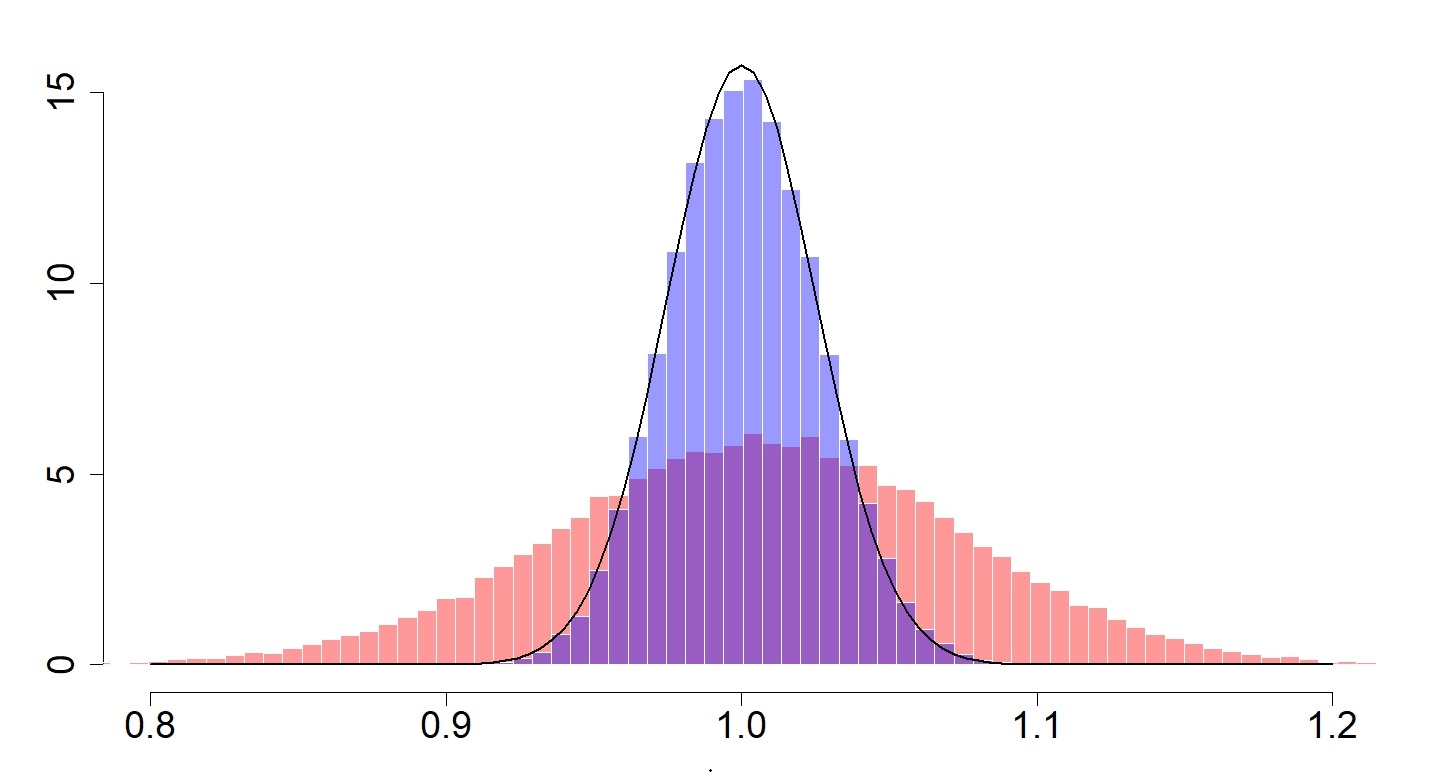}
		\hspace{0.5cm}
		\includegraphics[width=6cm]{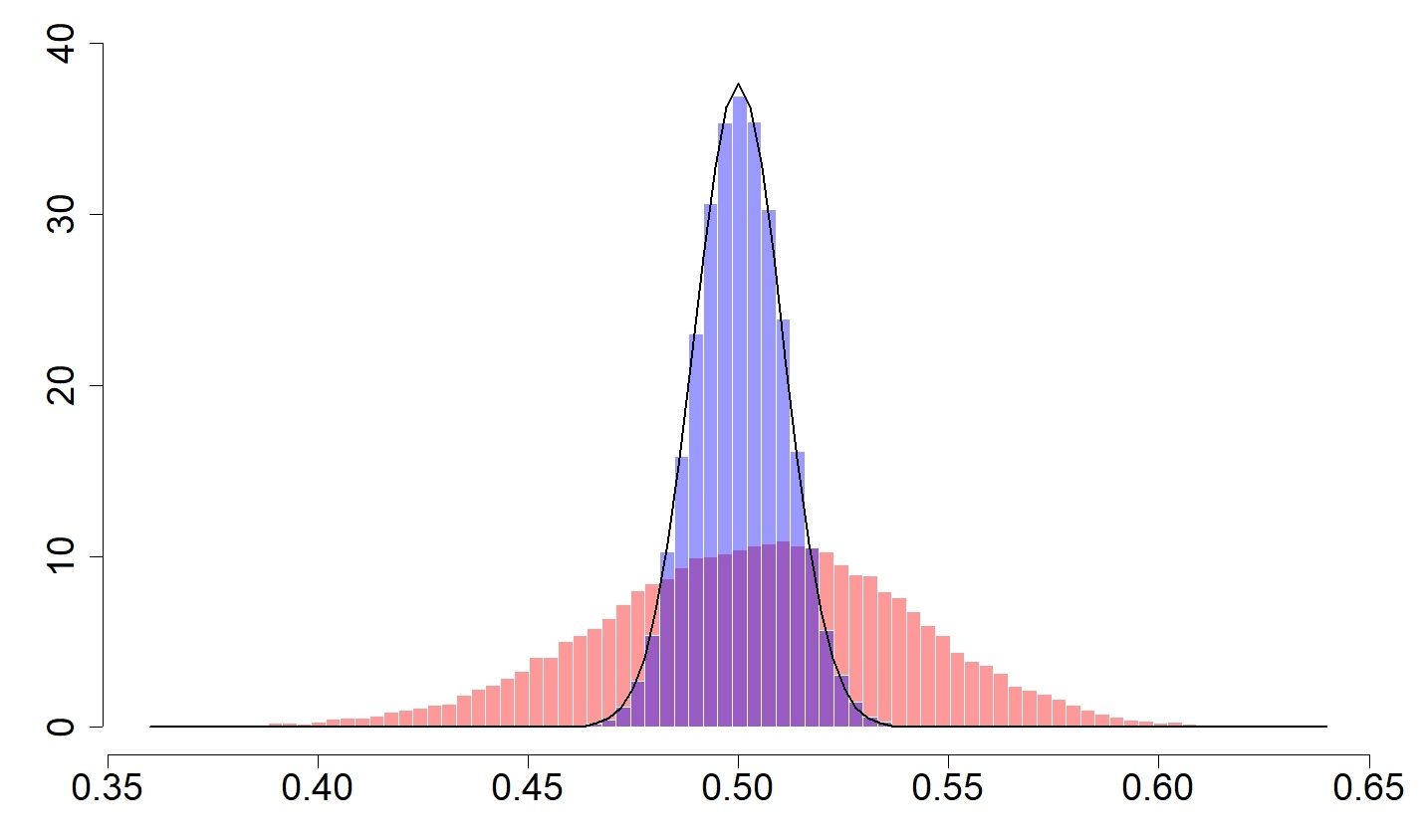}
		
		\caption{Histograms of $\hat{\alpha}$ for $\alpha=1$ (left) and $\alpha=1/2$ (right). The solid line shows the density of the asymptotic normal distribution. For comparison, histograms of the simpler moment estimators based on \eqref{moments} with $p=2$ and $p=4$ are included, exhibiting larger variances.}
		\label{fig:2}
	\end{figure}
	We conduct Monte Carlo simulations and generate synthetic data \(Y_i = X_{i/n} + \epsilon_i\), \(0 \le i \le n\), where \(\epsilon_i \overset{\text{i.i.d.}}{\sim} \Gamma(\alpha,\beta)\) with \(\beta = 2\) and \(\alpha \in \{1, 1/2\}\). We consider a simple semimartingale \(X_t = X_0 + 5 B_t\), with a standard Brownian motion $(B_t)$, where the initial value \(X_0\) drops out in the increments. More realistic specifications with stochastic volatility, leverage effects, and drift do not materially affect the results, since our target of inference is \((\epsilon_i)\) rather than \((X_t)\), so that only the noise-to-signal ratio is relevant. All results are based on \(50{,}000\) Monte Carlo replications. Code to reproduce all simulation results is available in a GitHub repository.\,\footnotemark It includes the numerical computation of the $\operatorname{AVAR}$ function used to generate Figure \ref{fig:1} as well as the numerical computation of the estimators. The plotted points in Figure \ref{fig:1} provide a comparison with empirical variances obtained from these Monte Carlo simulations of $\hat{\alpha}$, confirming good finite-sample agreement between our theoretical expression for $\operatorname{AVAR}$ and the empirical variances.
	
	The histograms in Figure \ref{fig:2} show that the empirical distributions of Monte Carlo samples of $\hat{\alpha}$ are well approximated by their asymptotic normal distributions. They also include a comparison with the method-of-moments estimator based on \eqref{moments} with $p=2$ and $p=4$, illustrating that the variance of our proposed estimator is considerably smaller. We use a sample size of \(n = 50{,}000\), which is realistic in view of the large number of quotes available for liquid stocks reported in Table~\ref{table}. 
	\section{Empirical Evidence}
	\label{sec:5}
	\footnotetext{\url{https://github.com/adriangrueber/LOMNgamma}}We consider NASDAQ limit order book data from the LOBSTER database\footnote{\url{https://lobsterdata.com}}, which has been used in several recent research papers including \cite{jumps} and \cite{MMN2}. The data has at least millisecond precision and no data cleaning procedures need to be performed before applying our statistics. As is standard for financial data, we consider log-prices as our observations $Y_i$.
	
	\begin{figure}[t]
		\begin{framed}
			~\includegraphics[width=0.3\textwidth]{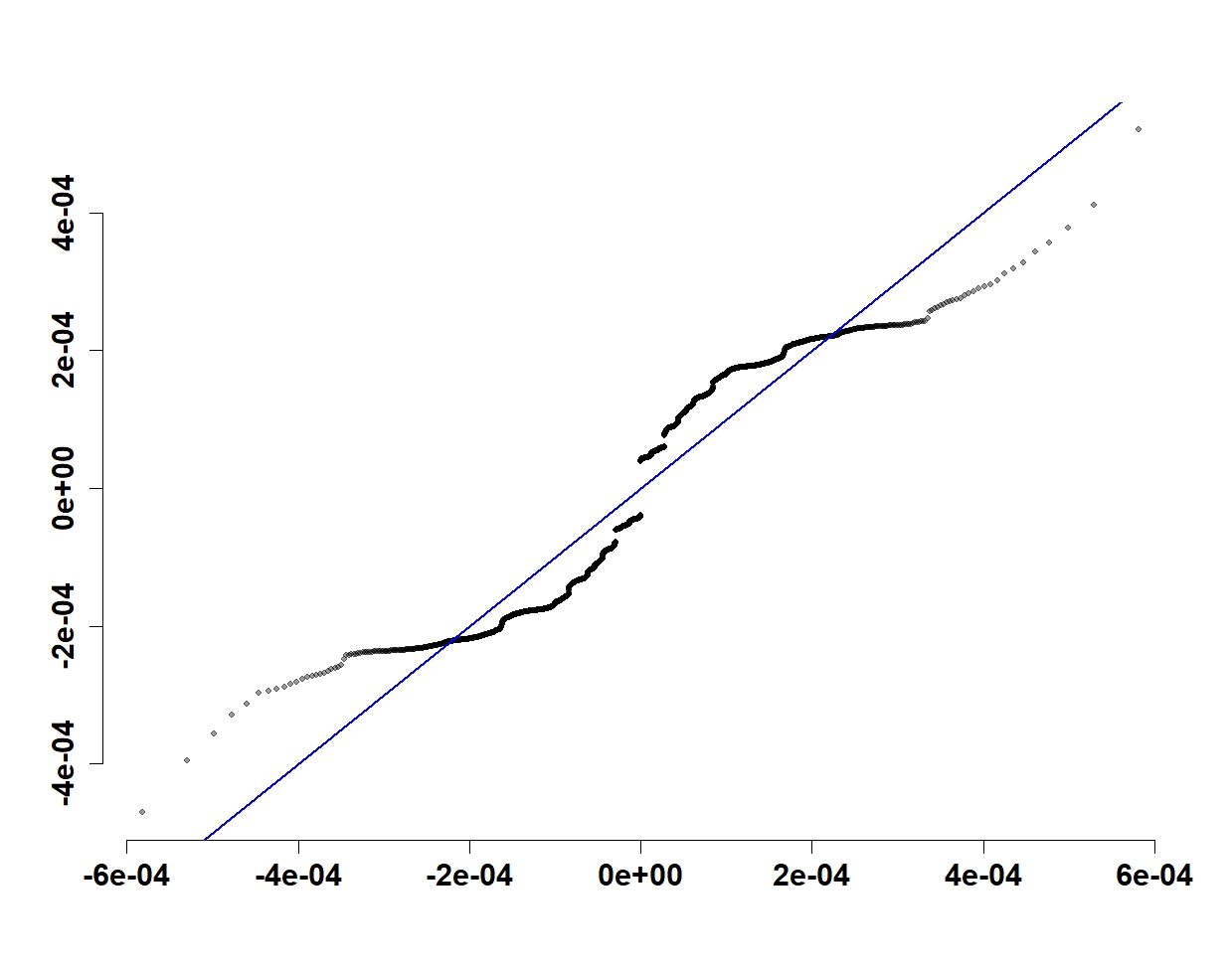} ~\includegraphics[width=0.3\textwidth]{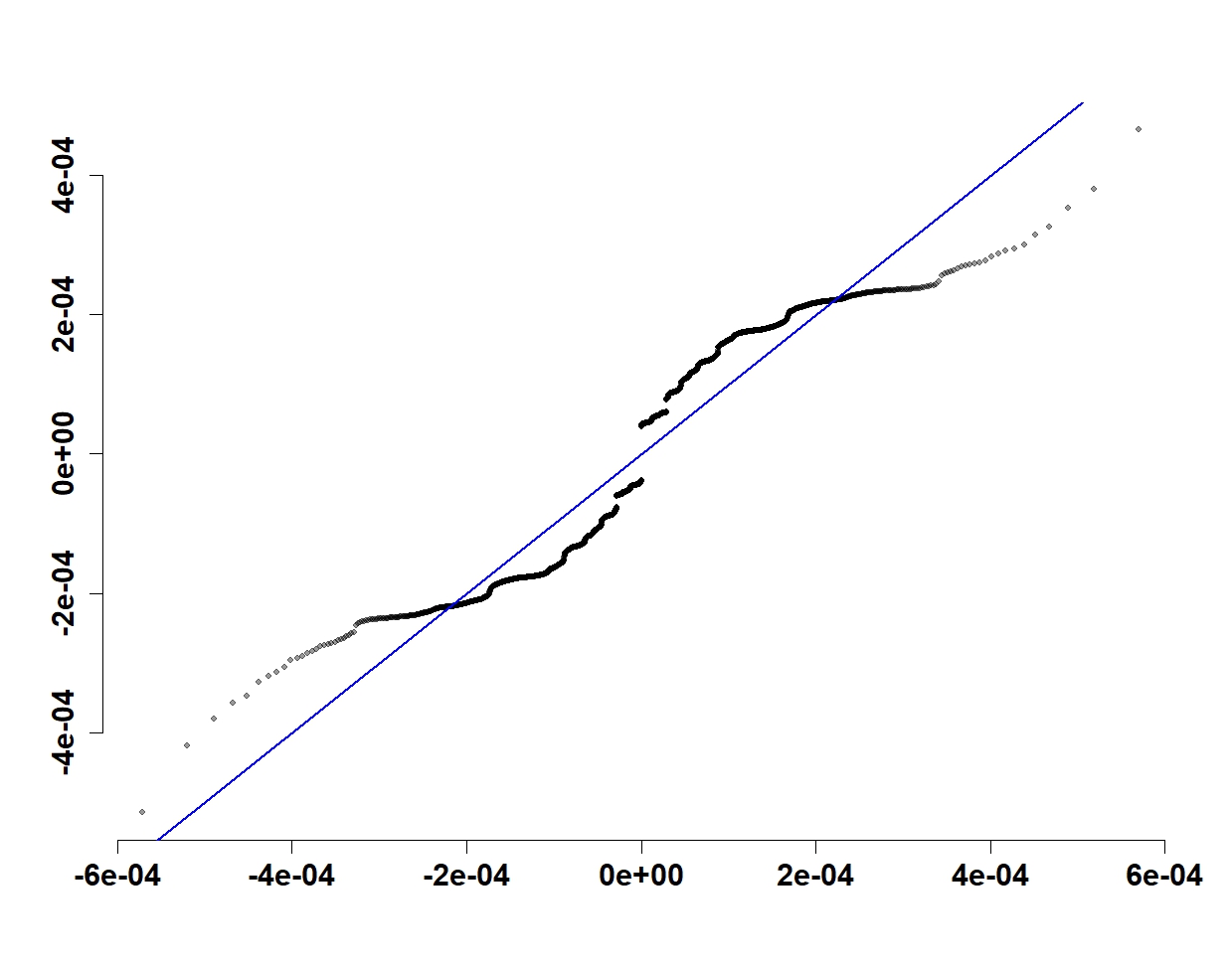}~\includegraphics[width=0.3\textwidth]{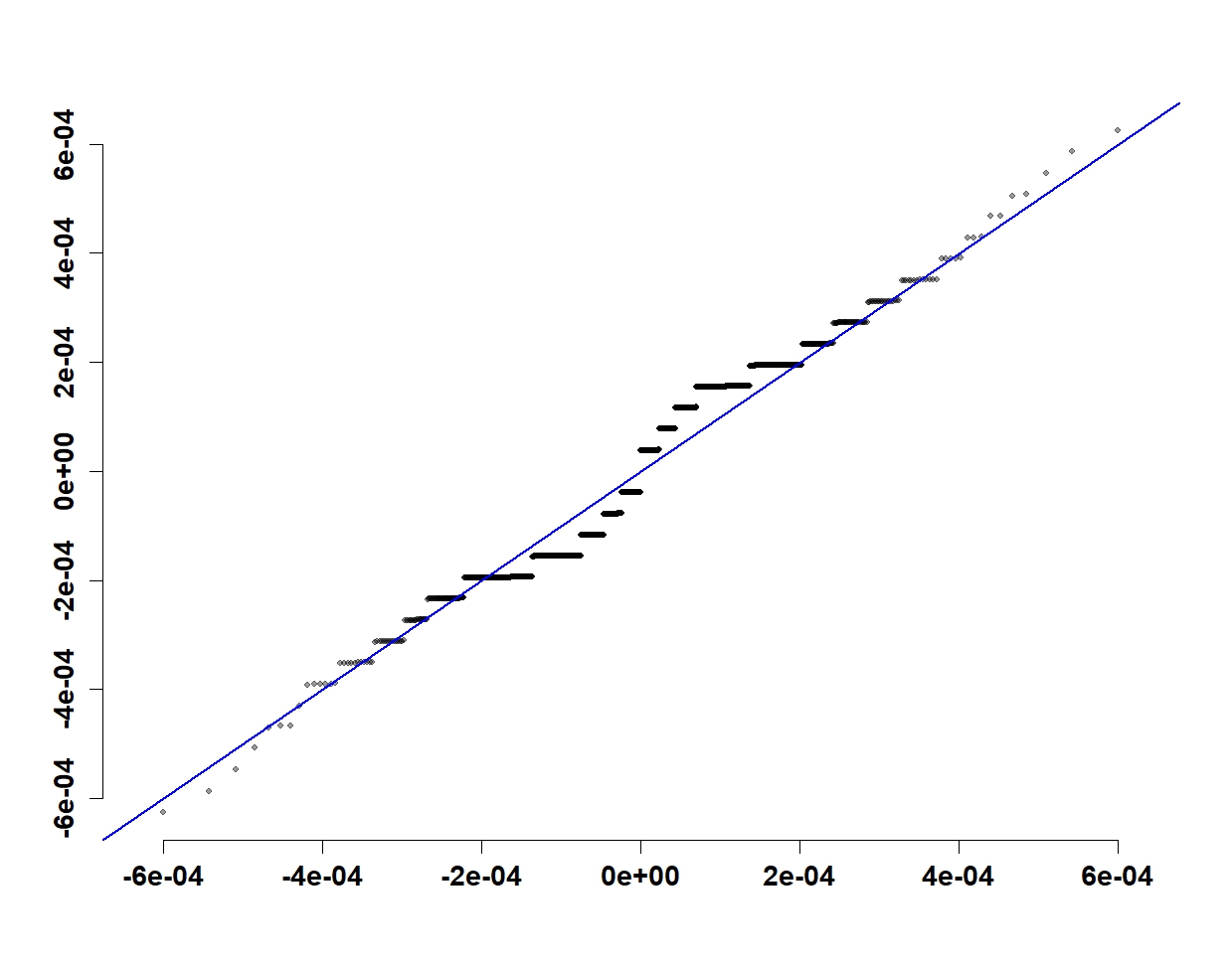}
		\end{framed}
		\caption{QQ-plots of empirical quantiles of $(Y_i-Y_{i-1})$ against quantiles of $\Gamma(\hat\alpha,\hat\beta)$ based on ask quotes (left) and bid quotes (middle) over the year 2024 and ask quotes over one day 2024-12-24 (right).\\ The plot is for an equally spaced grid of quantile levels $k/10000, k=1,\ldots,9999$.}
		\label{fig:3} 
	\end{figure}
	The QQ-plots in Figure \ref{fig:3} confirm overall adequate fits of the high-frequency log-returns of best ask and best bid quotes for the Apple stock (AAPL) by the Gamma difference distribution. We consider all quote revisions, excluding zero returns. The left panel uses best ask quotes over the full year 2024, while the middle panel is for best bid quotes over the same period. The right panel presents the fit for a single trading day, 2024-12-24. While such plots admittedly vary across individual days, they typically show better fits than those based on the full-year data. The step-like behaviour in the centre of the right plot is expected, as high-frequency log-returns at very fine time resolutions move on a discrete grid determined by the minimum tick size. This does not undermine the usefulness of models based on absolutely continuous distributions for inference on volatility and jumps of $(X_t)$, analogously to the MMN framework. Due to this effect and since the QQ-plots are based on a uniform grid of probability levels, multiple quantiles near the centre coincide. Overall, the symmetry, the bulk of the distribution, and the extreme quantiles are well captured by the two-parameter Gamma family. Results for other stocks are qualitatively similar.
		
	\begin{figure}[t]
		\includegraphics[width=\textwidth]{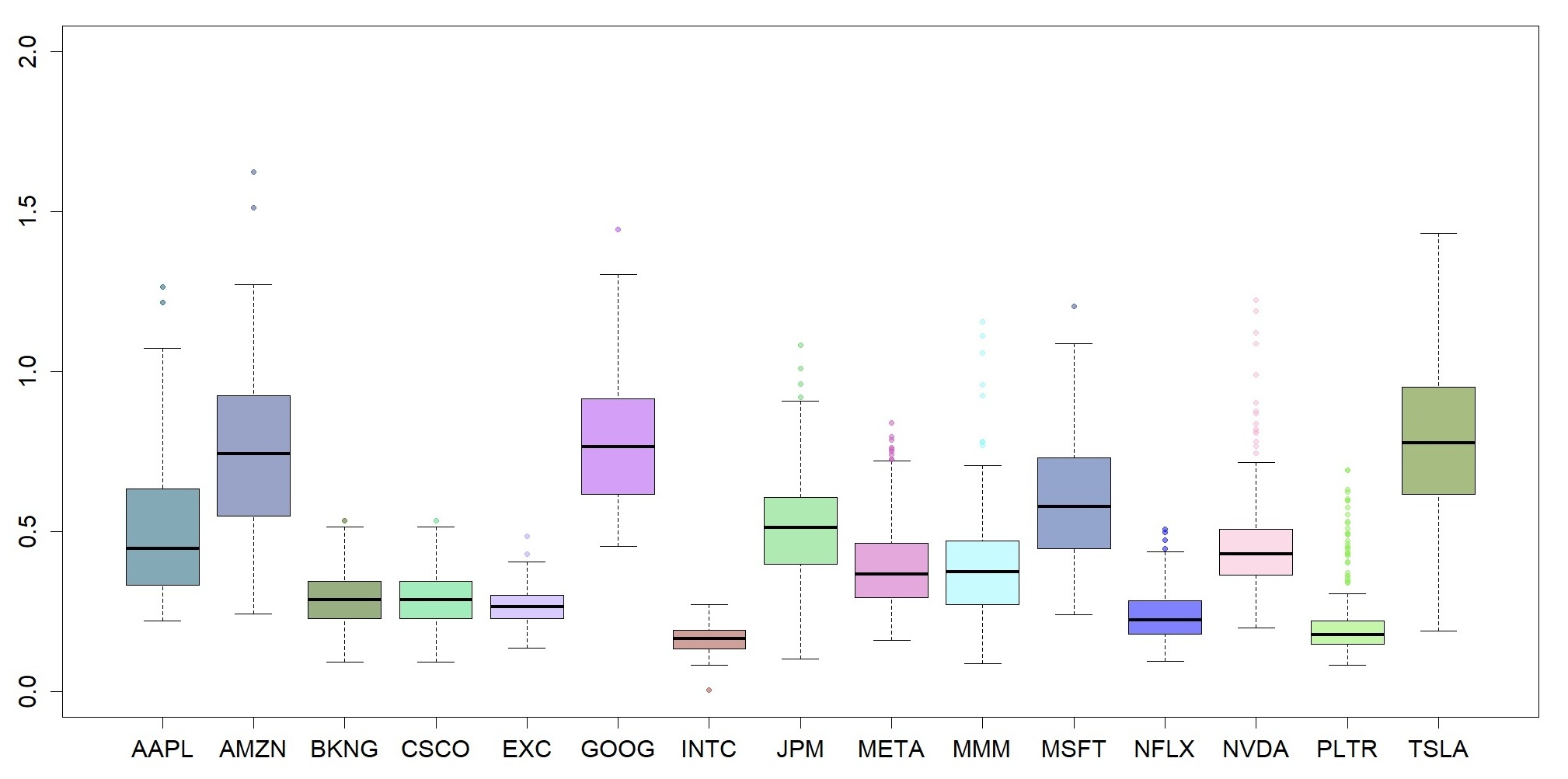}
		\caption{Boxplots of daily estimates $\hat\alpha$ over the year 2024 for 15 stocks.}
		\label{fig:4}
	\end{figure}
\renewcommand*{\arraystretch}{1.1}
\begin{table}[t]
	\centering
	\caption{Liquidity compared to estimated noise moments across 15 stocks.}
	\label{table}
	\resizebox{\textwidth}{!}{
		\begin{tabular}{|c|c c c c c c c c|}
			\hline
			Stock & AAPL & AMZN & BKNG & CSCO & EXC & GOOG & INTC & JPM \\
			\hline
			Avg.\ no.\ daily transactions & 108556 & 88386 & 2495 & 19487 & 9104 & 69258 & 37216 & 16233 \\
			
			Avg.\ no.\ daily subm.\ ask quotes & 404623 & 351833 & 7403 & 75712 & 43199 & 488848 & 159259 & 71978 \\
			
			Avg.\ no.\ daily subm.\ bid quotes & 411083 & 355412 & 7856 & 75420 & 43647 & 481941 & 159633 & 70699 \\
			
			Estimated first moment $\mathds{E}[\epsilon_i]$ & 0.0399 & 0.0606 & 0.1789 & 0.0794 & 0.1229 & 0.0701 & 0.0906 & 0.0666 \\
			\hline
			Stock & META & MMM & MSFT & NFLX & NVDA & PLTR & TSLA & \\
			\hline
			Avg.\ no.\ daily transactions & 34515 & 8320 & 55730 & 13163 & 199649 & 32129 & 155269 & \\
			
			Avg.\ no.\ daily subm.\ ask quotes & 84152 & 26393 & 204009 & 32764 & 533923 & 173343 & 409850 & \\
			
			Avg.\ no.\ daily subm.\ bid quotes & 83873 & 25774 & 214403 & 33371 & 538996 & 170223 & 423591 & \\
			
			Estimated first moment $\mathds{E}[\epsilon_i]$ & 0.0582 & 0.1229 & 0.0327 & 0.0809 & 0.0594 & 0.0974 & 0.0668 & \\
			\hline
		\end{tabular}
	}
\end{table}
\renewcommand*{\arraystretch}{1.0}
	
	The boxplots in Figure \ref{fig:4} are based on daily estimates of $\hat\alpha$ over the full year 2024 across 15 stocks, including the so-called \emph{Magnificent 7}--AAPL, AMZN, GOOG, META, MSFT, NVDA and TSLA--as well as eight additional blue-chip stocks from different sectors with varying liquidity. The variation in the estimates, as reflected by the box lengths, is driven less by the variance of $\hat\alpha$ and more by time-varying noise tail behaviour. Nevertheless, the estimates are remarkably stable for many stocks. The main empirical insights are:
	\begin{itemize}
	\item The noise tail parameter is asset-specific, since several estimates differ significantly across assets.
	\item Most noise tail parameters are significantly below the standard value of 1, suggesting that volatility and jump estimation can be performed even more efficiently than in \cite{BJR}, \cite{spot} and \cite{jumps}.
	\end{itemize}
	We investigate whether the empirical relations observed within the MMN framework, namely that more liquid stocks tend to exhibit lower noise and lower noise-to-signal ratios as shown in \cite{MMN}, have corresponding effects in the LOMN framework. Thereto, we relate liquidity measures, i.e., the number of transactions and order submissions, to the estimates of $\alpha$ shown in Figure \ref{fig:4} and to noise moments that depend on both parameters, $\alpha$ and $\beta$. Results are summarized in Table~\ref{table}. 
	
	The estimated first noise moments are comparatively small for the Magnificent 7, suggesting that higher liquidity may be associated with lower noise also in the LOMN framework. At the same time, some stocks with moderate liquidity exhibit relatively small estimated noise tail parameters in Figure \ref{fig:4}, indicating that fast rates for inference on the efficient price remain attainable. Future empirical work should therefore deepen the analysis of how one-sided noise characteristics relate to other asset-specific quantities.
	
	
	
	\section{Conclusion}
	We introduce a parametric framework for inferring a noise tail parameter of one-sided noise that captures the microstructure of limit order quotes. Our empirical results demonstrate the accuracy of the model and statistical methods and reveal asset-specific noise tail parameters. These are typically below 1, thereby enabling fast rates of inference on the efficient price and its characteristics.
	
	The implications for volatility and jump estimation mentioned in the introduction remain to be established in future work. This includes adaptive inference on the efficient price with pre-estimated noise tail parameter. Nonparametric methods for inferring the noise tail parameter are challenging but also of great interest, and their empirical results should be compared with the parametric ones presented in this paper. 

	\clearpage
	
	\bibliographystyle{apalike}
	\bibliography{references}
	\addcontentsline{toc}{section}{References}

\end{document}